\documentclass{amsart}

\usepackage{amsmath,amsthm,amssymb}
\usepackage{amsmath,amsfonts,amssymb,amstext,amsthm,amscd,mathrsfs,amsbsy}
\usepackage{graphicx}
\usepackage{url}

\newcommand{\CC}{\mathbb{C}}
\newcommand{\Cc}{\mathbb{C}^*}
\newcommand{\CP}{\mathbb{CP}}
\newcommand{\RR}{\mathbb{R}}
\newcommand{\ZZ}{\mathbb{Z}}

\newcommand{\nn}{\mathbf{n}}
\newcommand{\mm}{\mathbf{m}}

\newcommand{\OO}{\mathcal{O}}

\newcommand{\ol}{\bar}

\newtheorem{theorem}{Theorem}
\newtheorem{lemma}[theorem]{Lemma}
\newtheorem{prop}[theorem]{Proposition}

\title{Spaces of morphisms from a projective space to a toric variety}
\author{Jacob Mostovoy}
\author{Erendira Munguia-Villanueva}


\begin{document}

\bibliographystyle{plain}

\maketitle

\begin{abstract}
In this paper we study the space of morphisms from a complex projective space to a compact smooth toric variety $X$. It is shown that the first author's stability theorem for the spaces of rational maps from $\CP^m$ to $\CP^n$ extends to the spaces of continuous morphisms from $\CP^m$ to $X$ essentially, with the same proof. In the case of curves, our result improves the known bounds for the stabilization dimension.
\end{abstract}

\section{Introduction}

Given two complex algebraic varieties one can ask how well the space of continuous algebraic morphisms between them approximates the space of all corresponding continuous maps. This question was posed by G.\ Segal in \cite{SE79} where he studied the spaces of holomorphic maps from a Riemann surface to a complex projective space. He proved the following result. Let $S$ be a Riemann surface of genus $g$, $F_d(S,\CP^n)$ the space of rational functions of degree $d$ from $S$ to $\CP^n$, and $M_d(S,\CP^n)$ the corresponding space of continuous maps.
Similarly, denote by $F_d^*(S,\CP^n)$ and $M_d^*(S,\CP^n)$ the respective spaces of basepoint-preserving maps. 
\begin{theorem}[\cite{SE79}]
The inclusions 
$$F_d(S,\CP^n) \rightarrow M_d(S,\CP^n)\quad\text{and}\quad
F_d^*(S,\CP^n) \rightarrow M_d^*(S,\CP^n)$$
are homology equivalences up to dimension $(d-2g)(2n-1)$ for all $g\geq 0$. When $g=0$ these maps are also homotopy equivalences up to the same dimension.
\end{theorem}
Segal conjectured that his results could be generalized much further to include spaces of curves in algebraic varieties more general than projective spaces. Such generalizations were obtained in subsequent works by many authors: Kirwan \cite{KI86}, Gravesen \cite{GR89}, Guest \cite{GU84, GU95}, Mann and Milgram \cite{MM91, MM93},  Hurtubise \cite{HU96}, Boyer, Hurtubise and Milgram \cite{BHM01}, Cohen, Lupercio and Segal \cite{CLS99}, and others. At the moment, the strongest results for the case when $g=0$ (rational curves) and the target is a finite-dimensional variety $X$ are those of Guest \cite{GU95} for the case when $X$ is a (possibly singular) toric variety and that of Boyer, Hurtubise and Milgram \cite{BHM01} for the case when $X$ is a smooth K\" ahler manifold with a holomorphic action of a connected complex solvable Lie group which has an open dense orbit on which the group acts freely. A very general conjecture on this subject is given in the paper of Cohen, Jones and Segal \cite{CJS00}. Several papers \cite{GKY99, MO01} give real versions of Segal's theorem (the first real version being due to Segal himself in \cite{SE79}). 

Segal's theorem can also be generalized to the case when the domain of the maps has dimension greater than 1. Using Vassiliev's simplicial resolutions, the first author in \cite{MO06} and \cite{MO11} proved an analog of Segal's theorem for the spaces of continuous rational maps from $\CP^m$ to $\CP^n$ for $1\leq m\leq n$ (see also a real version in \cite{AKY11}).

The purpose of this note is to show that the proof given in  \cite{MO06, MO11} extends to the spaces of maps from $\CP^m$ to a smooth compact toric variety. We shall use the homogeneous coordinates for a toric variety $X$ and the description of the morphisms $\CP^m \rightarrow X$ given by Cox in \cite{COFu95} and \cite{COHo95}: a morphism $\CP^m \rightarrow X$ is given in these coordinates by a collection of homogeneous polynomials with certain properties. Let $P_f$ be the space of morphisms from $\CP^m$ to $X$ given by polynomials of fixed degrees, and whose restriction to a fixed hyperplane $\CP^{m-1}$ coincide with a given morphism $f: \CP^{m-1} \rightarrow X$. 
Consider the space of all continuous maps from $\CP^{m}$ to $X$ whose restriction to $\CP^{m-1}$ coincide with $f$. This space is homotopy equivalent to the $2m$-fold loop space of $X$ and in what follows will be denoted by $\Omega^{2m} X$. We shall show that the inclusion
\[ P_f \rightarrow \Omega^{2m} X\]
induces isomorphism in homology groups up to some dimension depending on $m$, $X$ and the degrees of the polynomials contained in $P_f$. The precise statement of our theorem will be given in the next section. 

In the case $m=1$ our result includes the homology part of the theorem of Guest (for smooth varieties only) as a particular case and,  moreover, gives an improvement over the results of \cite{GU95} in what concerns the stabilization dimension.

\section{Toric varieties}

A toric variety $X$ is a complex algebraic variety which admits an action of the algebraic torus $(\Cc)^n$, with an open and dense orbit on which the action is free.

Given  such $X$, we shall use the following notations: 
\begin{itemize}
\item $N\simeq \ZZ^n$  --- a lattice in $\RR^n$.
\item $M$ --- the integer dual of $N$.
\item $\langle\, , \rangle: N\otimes M\to\ZZ$ --- the natural pairing.
\item  $\Sigma\subset N_{\RR}=\RR^n$ --- the fan of $X$.
\item $\rho_1,...,\rho_r$ --- the one-dimensional cones in $\Sigma$. 
\item $\nn_i$ --- the primitive generator of $\rho_i \cap N$.
\item $D_i$ --- the divisor on $X$ corresponding to $\rho_i$. (Recall that the one-dimensional cones of $\Sigma$ are in correspondence with irreducible $(\Cc)^n$-invariant Weil divisors on $X$.)
\item $A_{n-1}(X)$ --- the group of Weil divisors modulo the subgroup of principal divisors. Two divisors $\sum a_i D_i$ and $\sum b_i D_i$ are in the same class in $A_{n-1}(X)$ if and only if there exists $\mm \in M$ such that $a_i=\langle \nn_i, \mm \rangle + b_i$ for $i=1,...,r$. 
\end{itemize} 
The details of the theory of toric varieties can be found in \cite{COFu95, COHo95, CLS11, FU93}. 

\medskip

In what follows we shall make extensive use of the homogeneous coordinate ring of $X$. It is defined as
$$R=\CC[x_1,...,x_r],$$ 
where the variable $x_i$ corresponds to $\rho_i$. Each monomial $\prod_i x_i^{a_i}$ determines a divisor $$D=\sum_i a_i D_i;$$  we shall use the notation 
$$x^D:=\prod_i x_i^{a_i}.$$ We say that the monomial $\prod_i x_i^{a_i}$ has degree $[D] \in A_{n-1}(X)$; it follows that two monomials $\prod_i x_i^{a_i}$ and $\prod_i x_i^{b_i}$ have the same degree if and only if there is some $m \in M$ such that $a_i=\langle \nn_i, m \rangle + b_i$. 

If we set
\[ R_{\alpha}=\bigoplus_{\deg(x^D)=\alpha} \CC \cdot x^D, \]
then the ring $R$ is $A_{n-1}(X)$-graded: 
\[ R=\bigoplus\, R_{\alpha}.\]
There is an isomorphism $$R_{\alpha} \simeq H^0(X,\OO_X(D)),$$ where $\OO_X(D)$ is the coherent sheaf on $X$ determined by the Weil divisor $D$.

In what follows we shall assume that $X$ is a compact and smooth toric variety. Each cone $\sigma \in  \Sigma$ determines the monomial $x^{\sigma}=\prod _{\rho_i \notin \sigma} x_i$.  Define the subvariety
 $$Y =\{(x_1,...,x_r) \, |\,  x^{\sigma}=0 \text{ for all } \sigma \in \Sigma \} \subset \CC^r,$$
 and the group
 $$G=\{ (\mu_1,....,\mu_r) \in (\Cc)^r\,  |\,  \prod_{i=1}^r \mu_i ^{\langle \nn_i, m  \rangle}=1 \text{ for all } m \in   M\}.$$
The group $G$ acts on $\CC^r - Y$ by multiplication coordinatewise.
The quotient 
$$X=(\CC^r - Y)/G$$
is the toric variety associated to the fan $\Sigma$. 
This description gives us the homogeneous coordinates on $X$. In these coordinates the morphisms $\CP^m \rightarrow X$ can be described explicitly as follows.

Let $P_{i}:\CC^{m+1} \rightarrow \CC$ homogeneous polynomials for $i=1,...,r$ such that $P_i$ has degree $d_{i}$ with $\sum_{i}d_i\nn_i=0$, and $(P_1(x),...,P_r(x)) \notin Y$ for all $x \in \CC^{m+1}-\{0\}$. Then the $r$-tuple $(P_1,...,P_r)$ induces a morphism $f:\CP^m \rightarrow X$. Furthermore,
    \begin{itemize}
     \item two sets of polynomials $\{ P_i\}$ and $\{ P'_i\}$ determine the same morphism if and only if there is $g \in G$ such that $(P_i)=g \cdot (P'_i)$;
     \item all morphisms $\CP^m \rightarrow X$ arise in this way.
    \end{itemize}

We say that a morphism $\CP^m \rightarrow X$ has degree $\ol{d}=(d_1,...,d_r)$ if it is given by an  $r$-tuple of homogeneous polynomials $(P_i)$ with $\deg P_i =d_i$. Let $P_f(\ol{d})$ be the set of all morphisms $\CP^m \rightarrow X$ of degree $\ol{d}$.

A set of edge generators $\{\nn_{i_1},...,\nn_{i_j}\}$  for the fan $\Sigma$ is \textit{primitive} if it does not lie in any cone of $\Sigma$ but every proper subset does. 
 It can be proved that 
\[Y=\bigcup_{\{\nn_{i_1},...,\nn_{i_j} \}  \text{ primitive}}
\{(x_1,...,x_r)\, |\, x_{i_1}=...=x_{i_j}=0 \}. \]

We can now state the main result of the present paper.

\begin{theorem}\label{main} Let $X$ be a smooth compact toric variety associated to a fan $\Sigma$, with $k$ the cardinality of the smallest primitive set of edge generators for $\Sigma$. Then, for all $m<k$, the inclusion of $P_f(\ol{d})$  into  $\Omega^{2m}X$  induces isomorphisms in homology groups for all  dimensions smaller than $d(2k-2m-1)-1$, where $\ol{d}=(d_1,...,d_r)$ and $d=\min\{ d_i\}$.
\end{theorem}

In the case of curves, that is, $m=1$, our estimate is, in general, much stronger than the theorem of Guest \cite{GU95} and the result of Boyer, Hurtubise and Milgram \cite{BHM01} specialized to the case of toric varieties, both of which give the stabilization dimension $d$.

The strategy of the proof is exactly the same as in \cite{MO06} and \cite{MO11} and mirrors other applications of Vassiliev's method. While it contains no essential novelty as compared to the case of a projective space \cite{MO11}, we find it necessary to go through the whole proof in some detail, since the case of a general smooth toric variety involves a number of additional features.

First, we construct a sequence of finite-dimensional approximations to the space of all continuous maps from  $\CP^n$ to a toric variety. These finite-dimensional approximations are spaces of maps that are given by polynomials in both holomorphic and antiholomorphic variables, of fixed degrees, and the first in the sequence is the space of holomorphic maps. The Stone-Weierstrass Theorem implies that these approximations indeed have the space of all continuous maps as a topological limit.

The second part of the proof consists in comparing the topology of two successive finite-dimensional approximations. Here we make use of the fact that these spaces are complements of discriminants in affine spaces. In particular, their topology can be related to the topology of the corresponding discriminants by the Alexander duality. In order to calculate the cohomology of these discriminants one resolves their singularities with the help of the Vassiliev's simplicial resolution, and considers the natural filtrations by the ``fibrewise skeleta''. 

The successive quotients of these filtrations are easy to describe explicitly for the first few terms. For higher terms, it suffices to give a dimension estimate which allows to discard them. One minor new observation in the present note is that in order to describe the relative effect of the stabilization maps on the topology of  the discriminants, one does not really need to write down the corresponding spectral sequence; this, of course, is not a significant simplification of the argument.

In the next section we construct the spaces of finite-dimensional approximations and show that the limit space is, indeed, homotopy equivalent to the space of all continuous maps. Then, in Section~\ref{cd} we describe the approximation spaces as complements of discriminants and study the topology of these discriminants.

\section{The Stone-Weierstrass Theorem}

A $(p,q)$-polynomial in the variables $z_i$ and $\bar{z}_i$ (where $0\leq i\leq m$) is a complex-valued homogeneous polynomial of degree $p$ in the ``holomorphic'' variables $z_i$ and degree $q$ in ``antiholomorphic'' variables $\bar{z}_i$. If $\ol{p}=(p_1,...,p_r)$, $\ol{q}=(q_1,...,q_r) \in \ZZ^r$ with $$\sum p_i \nn_i=\sum q_i \nn_i=0,$$ we define a $(\ol{p},\ol{q})$-map as a map $F:\CP^m \rightarrow X$ given by a $r$-tuple of $(p_i,q_i)$-polynomials such that $(F_1(x),....,F_r(x)) \notin Y$ for all $x \in \CC^{m+1}\setminus \{ 0 \}$.

Note that $(\ol{d},0)$-maps are the morphisms of degree $\ol{d}$. Two collections $(P_i)$ and $(P'_i)$ of $(p_i,q_i)$-polynomials determine the same morphism $F:\CP^m \rightarrow X$ if and only if there are functions $$g_1,\ldots ,g_r:\CC^{m+1}\setminus \{0\} \rightarrow \CC$$ such that $P_i=g_i P_i'$ and $(g_1(x),...,g_r(x)) \in G$ for all $x\in \CC^{m+1}\setminus \{0\}$. In particular, the $g_i$ do not vanish. Given 
\begin{equation} \label{a}
\ol{a}=(a_1,...,a_r) \in \ZZ^r
\end{equation} 
with $$\sum_i a_i \nn_i=0,$$ a $(\ol{p},\ol{q})$-map can be written as a $(\ol{p}+\ol{a},\ol{q}+\ol{a})$-map by multiplying all the polynomials by $(|z|^{2a_1},...,|z|^{2a_r})$ coordinatewise.

Let $f:\CP^{m-1} \to X$ be a $(\ol{p},\ol{q})$-map and assume that we have chosen the $(p_i,q_i)$-polynomials $f_i$ that define it. Denote by $W^i_{p_i,q_i}$ the complex affine space of all $(p_i,q_i)$-polynomials whose restriction to $\CP^{m-1}$ concides with $f_i$, and by $W_{\ol{p},\ol{q}}$ the cartesian product of the $W_{p_i,q_i}$. Let $P_f(\ol{p},\ol{q}) \subset W_{\ol{p}, \ol{q}}$ be the subspace of all the $r$-tuples in $W_{\ol{p},\ol{q}}$ whose values are not in $Y$ for all points in $\CC^{m+1}\setminus \{0\}$. These are precisely the $r$-tuples that define $(\ol{p},\ol{q})$-maps, though we stress that more than one $r$-tuple in  $P_f(\ol{p},\ol{q})$ may correspond to the same map.

For $\ol{a}$  as in (\ref{a}), we define the stabilization maps
\[P_f(\ol{p},\ol{q}) \rightarrow P_f(\ol{p}+\ol{a}, \ol{q}+\ol{a}),\]
\[(h_1(z),...,h_r(z)) \mapsto (|z|^{2a_1}h_1(z),...,|z|^{2a_r}h_r(z)).\]
The space $P_f(\ol{d}+\infty, \infty)$ is defined as the direct limit of these inclusions, where  $\ol{d}=\ol{p}-\ol{q}$.

In the next section we shall see that these maps induce isomorphisms in homology in the dimensions given in Theorem~\ref{main}. Here, we shall prove that the spaces $P_f(\ol{p},\ol{q})$ approximate the space of all continuous maps topologically:
\begin{prop}\label{limit} $P_f(\ol{d}+\infty, \infty)$ is homotopy equivalent to $\Omega^{2m}(X)$.
\end{prop}
The principal tool in the proof of this statement is the Stone-Weierstrass Theorem for vector bundles (see, for instance, \cite{MO06}):
\begin{theorem}
Let $E$ be a locally trivial real vector bundle over a compact space $B$, $s_\alpha:E \rightarrow B$ a set of its sections, and let $A$ be a subalgebra of the $\RR$-algebra $C(B)$ of continuous real-valued functions on $B$. Suppose that
\begin{itemize}
 \item the subalgebra $A$ separates points of $B$, that is, for any pair $x,y \in B$ there exists $h \in A$ such that $h(x) \neq h(y)$,
 \item for any $y \in B$ there exists $h \in A$ such that $h(y) \neq 0$,
 \item for any $y \in B$ the fibre of $E$ over $y$ is spanned by $s_{\alpha}(y)$.
 \end{itemize}
Then the $A$-module generated by the $s_{\alpha}$ is dense in the space of all continuous sections of $E$.
\end{theorem}
The Stone-Weierstrass Theorem has the following consequence: 
\begin{lemma}\label{aprox} Let $W$ be a finite $CW$-complex. Any continuous map $$F:\CP^m \times W \rightarrow X$$ can be uniformly approximated by a $(p,q)$-map, for some $p$ and $q$, whose coefficients are functions on $W$. Moreover, if the restriction of $F$ to $\CP^{m-1} \times \{x\}$ is represented by a collection of polynomials $f_{i,x}$, the approximating polynomials can be chosen so as to coincide with $f_{i,x}\cdot |z|^{2k}$, for some $k$, on $\CP^{m-1} \times W$.
\end{lemma}
The proof is entirely analogous to the proof in \cite{MO06} for the case $X=\CP^n$. Indeed, a continuous map $F:\CP^m \times W \rightarrow X$ can be thought as a family of maps $\CP^m \rightarrow X$ depending on a parameter $w\in W$. In particular, it can be given as a collection of $m+1$ sections $s_i$ of the line bundles $\OO(d_i)$, with $\sum_{\rho} d_{i} \langle \nn_{i}, \mm \rangle=0$  for some $\mm\in M$, and each $s_i$ depending on $w\in W$. Exactly as in \cite{MO06}, by the Stone-Weierstrass Theorem these sections can be approximated by functions, which, for any given $w$ are $(d_i+q, q)$-polynomials for some $q$, and these approximations can be chosen so as to give an arbirarly close approximation of $F$.

\begin{proof}[Proof of Proposition~\ref{limit}] For any compact Riemannian manifold $M$ and any space $E$ there exists $\epsilon>0$ such that any two maps $E\rightarrow M$ that are uniformly $\epsilon$-close, are homotopic. Thus, Lemma~\ref{aprox} implies that the map of the homotopy groups
\[\pi_kP_f(\ol{d}+\infty, \infty) \rightarrow \pi_k\Omega^{2m}(X)\]
is surjective for all $k$. Setting $W=S^k$ we get that any element of $\pi_k\Omega^{2m}(X)$ can be approximated by a class in $\pi_kP_f(\ol{d}+\infty, \infty)$. On the other hand, this isomorphism is also injective, as any homotopy that goes through continuous maps can be approximated by a homotopy through $(p,q)$-maps (set $W=S^k \times [0,1]$). Both spaces of maps have homotopy types of CW-complexes, so by the Whitehead Theorem they are homotopy equivalent.
\end{proof}

\section{The space $P_f(\ol{p},\ol{q})$ as a complement to a discriminant}\label{cd}

In what follows the notation $\CC^m$ will be used for the affine chart $z_m=1$ in $\CP^m$. We shall use the notation $T^{\bullet}$ for the one-point compactification of a space $T$.

Recall that we consider the toric variety $X$ as the quotient $(\CC^r-Y)/G$. The space $P_f(\ol{p},\ol{q})$ is the complement of a discriminant $\Theta$ in the space $W_{\ol{p},\ol{q}}$ which consists of collections of complex polynomials of multidegree $\ol{p}$ in the variables $z_i$ and $\ol{q}$ in the $\overline{z_i}$: 
\[ \Theta=\Theta_{\ol{p},\ol{q}}=\{ (F_1,...,F_r) \in W_{\ol{p},\ol{q}} \, |\, (F_1(z),...,F_r(z)) \in Y \text{ for some }  z \in \CC^{m}\}.\]
The  cohomology groups of $P_f(\ol{p},\ol{q})$ are related to the  homology groups of the one-point compactification of $\Theta$ by the Alexander duality: 
\[ \widetilde{H}^l(P_f(\ol{p},\ol{q}),\ZZ)\simeq
\widetilde{H}_{2N_{\ol{p},\ol{q}}-l-1}({\Theta}^{\bullet}, \ZZ)\] 
where $N_{\ol{p},\ol{q}}$ is the complex dimension of $W_{\ol{p},\ol{q}}$.

In order to study the topology of $\Theta^{\bullet}$, we shall use Vassiliev's method  of simplicial resolutions \cite{VA92}. We construct a simplicial resolution of $\Theta$ together with a natural filtration, and describe the first several terms of this filtration. Since not all the terms can be effectively described, we shall then use the truncation procedure as in \cite{MO11}.

Let $Z(\ol{p},\ol{q})=Z\subset W_{\ol{p},\ol{q}} \times \CC^m $ be the set $$Z=\{(F_1,...,F_r,x)\, |\, (F_1(x),...,F_r(x)) \in Y \}.$$ There is a projection map $Z \rightarrow \Theta$ which forgets the point $x\in \CC^m$.  We denote by  $Z^{\Delta}(\ol{p},\ol{q})$, or simply by $Z^{\Delta}$ the space of the non-degenerate simplicial resolution associated to this map. 
It is defined as the union of the spaces $Z_l=Z_l(\ol{p},\ol{q})$, with positive $l$, which are constructed as follows. First, embed (non-linearly) the space $\CC^m$ into some vector space $V$ in such a way that the image of any $2l$ distinct points of $\CC^m$ are not contained in any $2l-2$ dimensional subspace. This gives an embedding
\[Z\hookrightarrow  W_{\ol{p},\ol{q}} \times \CC^m \hookrightarrow  W_{\ol{p},\ol{q}}\times V.\]
The subspace $Z_l \subset W_{\ol{p},\ol{q}}\times V$ consists of all the $l-1$-simplices which lie in the fibres of the projection map onto $W_{\ol{p},\ol{q}}$ and whose vertices are on $Z_l$. We have a natural projection map $$Z_l\to \Theta$$ whose fibres are $(l-1)$-skeleta of  simplices of various, not necessarily finite, dimensions.  

\medskip

Consider the space of configurations of $l$ distinct points in $\CC^m$ with labels in $Y$: 
\[R_l:=\{(x_i,s_i) \in (\CC^m \times Y)^l\, |\, x_i\neq x_j  \text{ for } i\neq j \} \slash S_l\]
where the symmetric group $S_l$ acts by permuting the $l$ coordinates $(x_1,s_1)$, $\ldots,$  $(x_l,s_l)$. We write a point of $R_l$ as $\bigl((x_i,s_i)\bigr)$.
Recall that $Y$ can be seen as the union \[Y=\bigcup_{\{\nn_{i_1},...,\nn_{i_j} \}  \text{ primitive}}
\{(x_1,...,x_r)\, |\, x_{i_1}=...=x_{i_j}=0 \}, \]
where $\nn_1,...,\nn_r$ are generators for the 1-dimensional cones of the fan $\Sigma$ associated to $X$. Let $k$ be the cardinality of the smallest primitive set of edge generators of $\Sigma$. Then $R_l$ is a cell complex of dimension $2l(m+r-k)$.

Write $p$ for $\min\{ p_1,...,p_r\}$. We have the following
\begin{prop} \label{stages} For $l \leq p$ the space $Z_l \setminus Z_{l-1}$ is an affine bundle over the space $R_l$,  of real rank  $2(N_{\ol{p},\ol{q}}-rl)+l-1$.
\end{prop}

\begin{proof} 
A point in 
$$Z_l\setminus Z_{l-1} \subset W_{\ol{p},\ol{q}}\times V$$ 
can be written as $(F,t)$ where $F=(F_1,...,F_r)$ and $t$ is in the convex hull in $V$ of $l$ distinct points $x_1,...,x_l\in \CC^m$ with $F(x_j) \in Y.$
There is a projection 
 $$Z_l \setminus Z_{l-1} \rightarrow R_l$$ 
 $$(F,t) \mapsto \bigl( (x_j, F(x_j) )\bigr).$$ 
 The inverse image of a fixed point $\bigl((x_j,s_j)\bigr) \in R_l$ is the cartesian product of the space of $r$-tuples of $(p_i,q_i)$-polynomials with $$(F_1(x_j),...,F_r(x_j))=s_j$$ and the convex hull of $x_1,...,x_l$ in $V$. Each equation $F_i(x_j)=s_j$, with $x_j$ and $s_j$ fixed, is a linear equation in the space of coefficients $W_{p_i,q_i}$. If $l \leq p$, all these equations are affinely independent and, hence, the space of $r$-tuples of $(p_i,q_i)$-polynomials satisfying $F(x_j)=s_j$ has real dimension $2(N_{\ol{p},\ol{q}}-rl)$. The convex hull of the $x_j$ is an $l-1$ simplex; adding the dimensions we get the proposition. 
\end{proof}

This description of the filtration $Z_l$ allows us to write a spectral sequence converging to the (co)homology of $P_f(\ol{p},\ol{q})$; the building blocks for this spectral sequence are the (co)homology groups of the one-point compactifications of the spaces $R_l$ with suitably twisted coefficients. However, we do not need to construct this spectral sequence explicitly in order to obtain the stabilization theorem.

\medskip

The stabilization map that sends $P_f(\ol{p},\ol{q})$ to $P_f(\ol{p}+\ol{a},\ol{q}+\ol{a})$ is actually a restriction of a map $W_{\ol{p},\ol{q}}\to W_{\ol{p}+\ol{a},\ol{q}+\ol{a}}$ which also sends  $\Theta_{\ol{p},\ol{q}}$ to $\Theta_{\ol{p}+\ol{a},\ol{q}+\ol{a}}$. This map also gives a map of the corresponding simplicial resolutions of the discriminants which preserves the filtrations. 

\begin{prop}\label{bundle} For $l \leq p$ the space $Z_l(\ol{p}+\ol{a}, \ol{q}+\ol{a})$ is an orientable affine bundle over $Z_l(\ol{p}, \ol{q})$ of real rank $2(N_{\ol{p}+\ol{a},\ol{q}+\ol{a}}-N_{\ol{p},\ol{q}})$.
\end{prop}

\begin{proof} 
Let $\Delta_l(\CC^m\times Y)$ be the $l$th term of the filtration on the simplicial resolution of the map that collapses $\CC^m\times Y$ to one point. One can think of $\Delta_l(\CC^m\times Y)$ as the $(l-1)$-skeleton of the simplex whose vertices are points of $\CC^m\times Y$. 

There is a surjective map 
$$\pi : Z_l(\ol{p},\ol{q})\to\Delta_l(\CC^m\times Y),$$
which forgets the polynomials $F_1,\ldots, F_r$ but remembers the images of the points $x_j$. The stabilization map gives a commutative diagram
\[
\begin{array}{ccc}
Z_l(\ol{p},\ol{q})&\hookrightarrow&Z_l(\ol{p}+\ol{a},\ol{q}+\ol{a})\\
\pi \downarrow\quad&&\pi \downarrow\qquad\\
\Delta_l(\CC^m\times Y)&\simeq&\Delta_l(\CC^m\times Y)
\end{array}
\]
where the upper horizontal arrow is the stabilization and the lower horizontal equivalence is the self-homeomorphism induced by the map $\CP^m \times Y \to \CP^m \times Y$ given by $$(z,(s_1,...,s_r)) \mapsto (z, (|z|^{2a_1}s_1,...,|z|^{2a_r}s_r)).$$ Over each point of $\Delta_l(\CC^m\times Y)$ the fiber of $\pi$ in  $Z_l(\ol{p},\ol{q})$ is a complex affine subspace of the fiber of $\pi$ in $Z_l(\ol{p}+\ol{a},\ol{q}+\ol{a})$ whose complex codimension is exactly $N_{\ol{p}+\ol{a},\ol{q}+\ol{a}}-N_{\ol{p},\ol{q}}$.
\end{proof}

An immediate corollary of Proposition~\ref{bundle} is that for all $i$ we have a Thom isomorphism 
$$\widetilde{H}_i(Z_l(\ol{p},\ol{q})^{\bullet})\simeq \widetilde{H}_{i+2(N_{\ol{p}+\ol{a},\ol{q}+\ol{a}}-N_{\ol{p},\ol{q}})}(Z_l(\ol{p}+\ol{a},\ol{q}+\ol{a})^{\bullet})$$
for all $l\leq p$. 
\medskip

Now, a homotopy-equivalent space to the space $\Theta(\ol{p},\ol{q})^{\bullet}$ can be obtained from $Z_p(\ol{p},\ol{q})^{\bullet}$ via the construction of a truncated simplicial resolution as in \cite{MO11}. Essentially, the truncated, after the $p$th term, simplicial resolution consists in collapsing homotopically the non-contractible fibres of the natural projection 
$$Z_p(\ol{p},\ol{q})^{\bullet}\to\Theta(\ol{p},\ol{q})^{\bullet},$$
by means of adding a cone on each non-contractible fibre, see  \cite{MO11}. According to Lemma~2.1 of \cite{MO11}, the truncated simplicial resolution is obtained from  $Z_p(\ol{p},\ol{q})^{\bullet}$ by adding cells of dimension at most $\dim({Z_p(\ol{p},\ol{q})}\backslash {Z_{p-1}(\ol{p},\ol{q})})+1$. By Proposition~\ref{stages}, this dimension equals $$\left(2(N_{\ol{p},\ol{q}}-pr)+p-1\right) + \dim R_p+1=2N_{\ol{p},\ol{q}}+p(2m-2k+1).$$ 
Similarly, the corresponding truncated simplicial resolution for $\Theta(\ol{p}+\ol{a},\ol{q}+\ol{a})^{\bullet}$ can be obtained  by adding cells of dimension at most $2N_{\ol{p}+\ol{a},\ol{q}+\ol{a}}+p(2m-2k+1)$ to $Z_p(\ol{p}+\ol{a},\ol{q}+\ol{a})^{\bullet}$. In particular, we get an isomorphism
$$\widetilde{H}_i(\Theta_{\ol{p},\ol{q}}^{\bullet})\simeq \widetilde{H}_{i+2(N_{\ol{p}+\ol{a},\ol{q}+\ol{a}}-N_{\ol{p},\ol{q}})}(\Theta_{\ol{p}+\ol{a},\ol{q}+\ol{a}}^{\bullet})$$
for all $i>2N_{\ol{p},\ol{q}}+p(2m-2k+1)$. From the Alexander duality it follows that 
$$H^{i}(P_f(\ol{p},\ol{q}),\ZZ)=H^{i}(P_f(\ol{p}+\ol{a},\ol{q}+\ol{a}),\ZZ)$$
for all $i<p(2k-2m-1)-1$.
The fact that this isomorphism is induced by the stabilization map follows from the definition of the Alexander duality pairing as the linking number. If a map induces isomorphisms in integral cohomology up to some dimension, it also induces isomorphisms in integral homology in the same dimensions. 
We get 
\begin{prop} \label{stab}The morphism 
\[\widetilde{H}_i(P_f(\ol{p},\ol{q}),\ZZ) \rightarrow \widetilde{H}_i(P_f(\ol{p}+\ol{a},\ol{q}+\ol{a}),\ZZ)\] 
induced by the stabilization map is an isomorphism for all  $i<p(2k-2m-1)-1$.
\end{prop}

Since the direct limit of the stabilization maps is homotopy equivalent to the space of continuous maps by Proposition~\ref{limit}, Theorem~\ref{main} now follows.

\section*{Acknowledgements}
This work is part of the doctoral thesis of the second author and was supported by a CONACyT scholarship.

\end{document}